\RequirePackage[T1]{fontenc}
\documentclass[preprint,twoside,12pt]{amsart}

\usepackage{amsfonts}
\usepackage{amsmath}
\usepackage{amsthm}
\usepackage{amssymb}
\usepackage{bbm}
\usepackage{fancyhdr}
\usepackage{tikz}
\usetikzlibrary{positioning,arrows.meta}
\usepackage{etoolbox}
\usepackage{stmaryrd}
\usepackage[all]{xy}
\usepackage{physics}
\usepackage{tikz-cd}
\usepackage{mathrsfs}
\usepackage{mathtools}
\usepackage{xcolor}
\usepackage{latexsym}
\usepackage[normalem]{ulem}
\usepackage{thm-restate}
\usepackage{bussproofs}

\tikzcdset{row sep/normal=50pt, column sep/normal=50pt}

\usepackage[bookmarks=true,bookmarksdepth=3,bookmarksnumbered=true,colorlinks=false,pdfborder={0,0,0}]{hyperref} 

\makeatletter
\newcommand{\SafeTocLink}[3] {
\hyperlink{#1}{#2\hfill #3}
\patchcmd{\l@section}
{\@dottedtocline{1}{0em}{1.5em}{#1}{#2}}{\@dottedtocline{1}{0em}{1.5em}{\SafeTocLink[\@currentHref]}
{#1}{#2}}{} }{}{}
\makeatother

\theoremstyle{definition}
\newtheorem{theorem}{Theorem}[section]

\newtheorem{lemma}[theorem]{Lemma}
\newtheorem{corollary}[theorem]{Corollary}
\newtheorem{proposition}[theorem]{Proposition}
\newtheorem{definition}[theorem]{Definition}

\newtheorem{remark}[theorem]{Remark}

\newcommand{\id}{\mathstrut \mathrm{id} }

\newcommand{\cou}{\mathstrut \mathrm{co\mu} }

\newcommand{\dom}{\mathstrut \boldsymbol{\mathrm{dom}}\,}
\newcommand{\End}{\mathstrut \boldsymbol{\mathrm{End}}\,}
\newcommand{\cod}{\mathstrut \boldsymbol{\mathrm{cod}}\,}

\newcommand{\incl}{\mathstrut \boldsymbol{\mathrm{Incl}} }

\newcommand{\set}{\mathstrut \boldsymbol{\mathrm{set}} }

\newcommand{\mon}{\mathstrut \boldsymbol{\mathrm{mon}} }
\newcommand{\Pair}{\mathstrut \boldsymbol{\mathrm{Pair}} }
\newcommand{\Mod}{\mathstrut \boldsymbol{\mathrm{mod}} }

\newcommand{\coloreduline}[3][red]{
\tikz[baseline=(X.bese)]{\node[innee sep=0pt, outer sep=0pt](X) {\textcolor{#2}{#3}};\draw[#1, line width=0.6pt](X.south west) -- (X.south east); } }

\newcommand{\Thmref}[1]{\hyperref[#1] {\coloreduline[red]{black}{\ref*{#1}}}}

\title{\textbf{YOSHIDA ALGEBRA FOR GROUPOIDS}}

\author{KEITARO SHIIZUKA}

\thanks{Department of Mathematics, Kindai University, 3-4-1, Kowakae, Higashi-Osaka, Osaka 577-8502, Japan, Email: keitaro.shizuka.math@gmail.com , This research did not receive any specific grant from funding agencies in the public, commercial, or not-for-profit sectors.}

\begin{document}

\maketitle

\begin{abstract}
In this paper, we extend the notion of the Yoshida algebra of a finite group introduced in \cite{Yos83} to finite groupoids and investigate its fundamental properties. Our main results show that the center of the Yoshida algebra of a finite groupoid is isomorphic to the center of the corresponding groupoid algebra, and that there exists a surjective ring homomorphism from the crossed Burnside ring of a finite groupoid, introduced in \cite{Shi26+}, onto the center of the Yoshida algebra of a finite groupoid. 
\end{abstract}

\textbf{Key words and phrases} : Groupoid, groupoid action, Yoshida algebra, crossed Burnside ring, center, linear representation.

\tableofcontents

\section{Introduction}
The concept of a groupoid was introduced by H. Brandt in 1927 (see \cite{Bra27}). It generalizes the notion of a group, the main deference being that composition is not necessarily defined for every pair of elements. Various equivalent definitions of groupoids and their actions exist. In this manuscript, from a categorical perspective, we treat a groupoid $\mathcal{G}$ as a category in which every morphisms is invertible. Under this viewpoint, an action is given by a functor $\mathcal{G}\to \set$.
In Section $2$, we provide notations and basic definitions. In Section $3$, we study the ring structure of $\mathbbm{k}$-$\Mod^\mathcal{G}(\mathbbm{k}[S],\mathbbm{k}[S])$ and prove that its center naturally decomposes as a direct product. In Section $4$, we extend the definition of the $G$-set $\Omega$, which appears in the construction of the Yoshida algebra of a finite group, to the setting finite groupoids. We prove that the center of the Yoshida algebra of a finite groupoid is isomorphic to the center of the corresponding groupoid algebra.  
Finally, we construct a surjective ring homomorphism from the crossed Burnside ring of a finite groupoid, introduced \cite{Shi26+}, onto the center of the Yoshida algebra of a finite groupoid.
\section{Preliminaries}
\subsection{Notations and conventions}
Throughout this manuscript, all groupoids are assumed to be finite.
Let $\mathbbm{k}$ be a commutative ring, and denote by $\mathbbm{k}$-$\Mod$ the category of finitely generated $\mathbbm{k}$-modules. For a ring $A$, we write $Z(A)$ for its center. Other notation and conventions follow \cite{Shi26+}.

\subsection{Groupoids}
In this subsection, we prepare the groundwork for the subsequent sections. Throughout this paper, we treat groupoids within  categorical framework.

\begin{definition}\cite{Shi26+} {}\\
\noindent
$(1)$ A category $\mathcal{G}$ is a \textbf{groupoid} if every morphism in $\mathcal{G}$ is invertible.\\
\noindent
$(2)$ For each object $x\in\mathcal{G}_0$, the set $\mathcal{G}(x,x)$, denoted by $\mathcal{G}_x$, is called the \textbf{isotropy group} at $x\in\mathcal{G}_0$. \\
\noindent
$(3)$ A subcategory $\mathcal{H}$ of a groupoid $\mathcal{G}$  that is itself a groupoid is called a \textbf{subgroupoid} of $\mathcal{G}$. If $\mathcal{H}$ is a subgroupoid of $\mathcal{G}$, then the inclusion functor $\mathcal{H}\to \mathcal{G}$ is denoted by $\incl^\mathcal{G}_\mathcal{H}$.\\ 
\noindent
$(4)$ If $X$ is a finite set, then the groupoid $\Pair(X)$ is the category defined by the following conditions: 
\begin{enumerate}
\item $\Pair(X)_0:=X,$
\item $\Pair(X)(x,y):=\{(x,y)\}$, for \;$(x,y)\in X\times X,$
\item The composition of $f=(x,y)$ and $g=(y,z)$ is defined by $gf:=(x,z)$.
\end{enumerate}
The groupoid $\Pair(X)$ is called the \textbf{groupoid of pairs}.\\
\noindent
$(5)$ Let $\mathcal{G}$ be a groupoid. For $x,y\in \mathcal{G}_0$, define $x\sim y$ if there exists a morphism $g:x\to y$. Then $\sim$ is an equivalence relation on $\mathcal{G}_0$. The elements of the quotient set $\mathcal{G}_0/{\sim}$ are called the \textbf{connected components} of $\mathcal{G}$, and we denote this set by $\pi_0(\mathcal{G})$. In particular, if $|\pi_0(\mathcal{G})|=1$, then $\mathcal{G}$ is called a \textbf{connected groupoid}.\\
\noindent
$(6)$ The functor category $\set^\mathcal{G}$ is called the \textbf{category of $\mathcal{G}$-sets}. An object in $\set^\mathcal{G}$ is called a \textbf{$\mathcal{G}$-set}. 
\end{definition}

The following proposition is well known.
\begin{proposition}\label{gpdst}
Let $\mathcal{G}$ be a connected groupoid and let $x\in\mathcal{G}_0$. Then, $\mathcal{G}\cong\mathcal{G}_x\times\;\Pair(\mathcal{G}_0)$.
\end{proposition}

\begin{itemize}
\item A functor $X:\mathcal{G}\to \mathbbm{k}$-$\Mod$ is called a \textbf{$\mathbbm{k}$-linear representation of $\mathcal{G}$}. The functor category $\mathbbm{k}$-$\Mod^\mathcal{G}$ is called \textbf{the category of $\mathbbm{k}$-linear representations of $\mathcal{G}$}.
\item  Let $X$ be a $\mathcal{G}$-set. The \textbf{permutation representation} associated with $X$ is the functor  $\mathbbm{k}[X]:\mathcal{G}\to \mathbbm{k}-\Mod$ defined by the following conditions: 
\item For an object $x\in \mathcal{G}_0$, define $\mathbbm{k}[X](x):=\mathbbm{k}[X(x)]$
\item For a morphism $g:x\to y$, define $\mathbbm{k}[X](g):\mathbbm{k}[X(x)]\to \mathbbm{k}[X(y)]$ by extending $s\mapsto X(g)(s)$ $\mathbbm{k}$-linearly for $s\in X(x)$.
\end{itemize}

\begin{definition}
We define the algebra $\mathbbm{k}[\mathcal{G}]$ of $\mathcal{G}$ over $\mathbbm{k}$ to be the free $\mathbbm{k}$-module with the elements of $\mathcal{G}_1$ as a $\mathbbm{k}$-basis and with multiplication of two elements is given symbolically by 
\[(\sum_{g\in\mathcal{G}_1} x_gg)*(\sum_{h\in\mathcal{G}_1} y_hh):=\sum_{(g,h)\in\mathcal{G}^{(2)}} x_gy_h gh\]
Where, $\mathcal{G}^{(2)}:={\mathcal{G}_1}_{\dom}\times_{\cod}\mathcal{G}_1$.
\end{definition}

\begin{theorem}\label{gd}
If $\mathcal{G}\cong \coprod^n_{i=1}{\mathcal{G}}_{[i]}$ is the decomposition of $\mathcal{G}$ into its connected components, then the ring $Z(\mathbbm{k}[\mathcal{G}])$ decompose as \[Z(\mathbbm{k}[\mathcal{G}])\cong 
\prod^n_{i=1}Z(\mathbbm{k}[\mathcal{G}_{[i]}]).\]
\end{theorem}

\begin{proof}
By the connected decomposition $\mathcal{G}\cong \coprod^n_{i=1}{\mathcal{G}}_{[i]}$, there is a ring isomorphism between $\mathbbm{k}[\mathcal{G}]$ and $\prod^n_{i=1}\mathbbm{k}[\mathcal{G}_{[i]}]$. Therefore, \[Z(\mathbbm{k}[\mathcal{G}])\cong 
Z(\prod^n_{i=1} \mathbbm{k}[\mathcal{G}_{[i]}])\cong\prod^n_{i=1}Z(\mathbbm{k}[\mathcal{G}_{[i]}]).\]
\end{proof}

\begin{lemma}\label{morita1}
If $\mathcal{G}$ is connected and let $x\in\mathcal{G}_0$, then $\mathbbm{k}[\mathcal{G}]$ and $\mathbbm{k}[\mathcal{G}_x]$ are Morita equivalent.
\end{lemma}

\begin{proof}
For a finite category $C$, the equivalence between $k$-$\Mod^C$ and $k[C]$-$\Mod$ was shown by Barry Mitchell in \cite{Mit72}. By applying this result to a finite groupoid $\mathcal{G}$, we have two categorical equivalences $k$-$\Mod^\mathcal{G}\simeq k[\mathcal{G}]$-$\Mod$ and $k$-$\Mod^{\mathcal{G}_x}\simeq k[\mathcal{G}_x]$-$\Mod$. By Lemma 2.3 in \cite{Shi26+}, we have a categorical equivalence $\mathcal{G}\simeq \mathcal{G}_x$, and hence obtain a categorical equivalence $k$-$\Mod^{\mathcal{G}}\simeq k$-$\Mod^{\mathcal{G}_x}$. Therefore, $\mathbbm{k}[\mathcal{G}]$ and $\mathbbm{k}[\mathcal{G}_x]$ are Morita equivalent.\end{proof}

\begin{remark}
The correspondence underlying the categorical equivalence between $k$-$\Mod^C$ and $k[C]$-$\Mod$ is also described in Section 2.2 of \cite{Xu07}.
\end{remark}

\begin{corollary}\label{morita}
If $\mathcal{G}$ is connected and let $x\in\mathcal{G}_0$, then $Z(\mathbbm{k}[\mathcal{G}])\cong Z(\mathbbm{k}[\mathcal{G}_x])$ holds.
\end{corollary}

\begin{proof}
This follows from Lemma \ref{morita1}.
\end{proof}

\subsection{Crossed Burnside rings for groupoids}
In this subsection, we recall the definitions and basic properties of the crossed Burnside ring of a groupoid. For further details, we refer the reader to \cite{Shi26+}.

\begin{itemize}
\item If $\mathcal{G}$ is a groupoid, then \textbf{conjugation action} of $\mathcal{G}$
is the functor $\mathcal{G}^c:\mathcal{G}\to \mon$ defined by the following conditions: 
\begin{enumerate}
\item For an object $x\in \mathcal{G}_0$, define $\mathcal{G}^c(x):=\mathcal{G}_x,$
\item For a morphism $g:x\to y$, define $\mathcal{G}^c(g):\mathcal{G}_x\to \mathcal{G}_y$ by $a\mapsto gag^{-1}$.
\end{enumerate}
\item A functor $\mathcal{G}\to \mon$ is called a \textbf{$\mathcal{G}$-monoid}. If $X$ is a $\mathcal{G}$-monoid, and let $U:\mon\to\set$ be a forgetful functor. The composite functor $U\circ X$ is a $\mathcal{G}$-set. We denote $U\circ X$ by $\overline{X}$.
\end{itemize}

\begin{theorem}\cite[Theorem 3.1, Definition 3.3]{Shi26+}
If $S$ is a $\mathcal{G}$-monoid , then the category $\set^\mathcal{G}/ \overline{S}$ is a monoidal category. The multiplication on $K_0(\set^\mathcal{G}/\overline{S},\sqcup)$ is induced by the monoidal structure on $\set^\mathcal{G}/ \overline{S}$
\[(Y\stackrel{\theta}{\to} \overline{S})\otimes (Z\stackrel{\tau}{\to}\overline{S}):=(Y\times Z\stackrel{\theta\times\tau}{\to} \overline{S}\times \overline{S}\stackrel{m}{\to} \overline{S}).  \]
induced by the monoid multiplication. This ring is denoted by $B^c(\mathcal{G},\overline{S})$ and we denote  $B^c_\mathbbm{k}(\mathcal{G})=\mathbbm{k}\otimes_\mathbbm{Z} B^c(\mathcal{G},\mathcal{G}^c)$.
\end{theorem}

\begin{theorem}\cite[Theorem 3.8]{Shi26+}\label{cd}
If $\mathcal{G}\cong\coprod^n_{i=1}\mathcal{G}_{[i]}$ is the decomposition of $\mathcal{G}$ into its connected components, then two rings $B^c(\mathcal{G},\mathcal{G}^c)$ and $\prod_{x\in\pi_0(\mathcal{G})}B^c(\mathcal{G}_x)$ are isomorphic.
\end{theorem}

\begin{lemma}\cite[(3.1), (4.2)]{OY01}\label{hom}
If $G$ is a finite group, then the following two statements hold:\\
\noindent
$(1)$ The $\mathbbm{k}$-module $B^c_\mathbbm{k}(G)$ is free and generated by finitely many $G$-conjugacy classes $[H,a]_G$. Here $H$ is subgroup of $G$, and $a$ is an element of the centralizer $C_G(H)$. Two pairs $(H,a)$ $(K,b)$ are said to be $G$-conjugate if there exists an element $g\in G$ such that $H=gKg^{-1}, a=gbg^{-1}$.\\
\noindent
$(2)$ The map $B_\mathbbm{k}^c(G)\to Z(\mathbbm{k}[G])$ defined by $[H,a]\mapsto \sum_{x\in[G/H]}xax^{-1}$ is a surjective ring homomorphism.
\end{lemma}

\begin{theorem}\label{stt}
If $\mathcal{G}$ is connected and $x\in\mathcal{G}_0$ then there exists a surjective ring homomorphism from $B^c_\mathbbm{k}(\mathcal{G})$ to $Z(\mathbbm{k}[\mathcal{G}])$
\end{theorem}

\begin{proof}
By Theorem \ref{gd} and Theorem \ref{cd}, we obtain the isomorphism $B^c_\mathbbm{k}(\mathcal{G})\cong B^c_\mathbbm{k}(\mathcal{G}_x)$ and $Z(\mathbbm{k}[\mathcal{G}])\cong Z(\mathbbm{k}[\mathcal{G}_x])$ respectively. By Lemma \ref{hom}, there exists a surjective ring homomorphism from $B^c_\mathbbm{k}(\mathcal{G}_x)$ to $Z(\mathbbm{k}[\mathcal{G}_x])$. Therefore, there exists a surjective ring homomorphism  from $B^c_\mathbbm{k}(\mathcal{G})$ to $Z(\mathbbm{k}[\mathcal{G}])$.
\end{proof}

\section{The center of the endomorphism ring in $\mathbbm{k}$-$\Mod^\mathcal{G}$}
In this section, we prove certain basic properties of the ring $\mathbbm{k}$-$\Mod^{\mathcal{G}}(\mathbbm{k}[S],\mathbbm{k}[S])$ for a $\mathcal{G}$-set $S$.

\begin{theorem}
If $S$ is a $\mathcal{G}$-set, then $\mathbbm{k}$-$\Mod^{\mathcal{G}}(\mathbbm{k}[S],\mathbbm{k}[S])$ is a ring.
\end{theorem}

\begin{proof}
Since $\mathbbm{k}$-$\Mod$ is an additive category, so $\mathbbm{k}$-$\Mod^{\mathcal{G}}$ is also additive. Consequently, $\mathbbm{k}$-$\Mod^{\mathcal{G}}(\mathbbm{k}[S],\mathbbm{k}[S])$ becomes a ring, with the multiplication given by composition.
\end{proof}

\begin{theorem}\label{center}
If $\mathcal{G}$ is connected, $x\in \mathcal{G}_0$ and $S$ is a $\mathcal{G}$-set, then two rings $Z(\mathbbm{k}$-$\Mod^{\mathcal{G}}(\mathbbm{k}[S],\mathbbm{k}[S]))$ and $Z(\mathbbm{k}$-$\Mod^{\mathcal{G}_x}(\mathbbm{k}[S\circ\incl^\mathcal{G}_{\mathcal{G}_x}],\mathbbm{k}[S\circ\incl^\mathcal{G}_{\mathcal{G}_x}]))$ are isomorphic.
\end{theorem}

\begin{proof}
We set $S_x:=S\circ\incl^\mathcal{G}_{\mathcal{G}_x}$. Take a family of morphisms $\{g_y\in\mathcal{G}(x,y)\}_{y\in\mathcal{G}_0}$ such that $g_x=\id_x$. We define the map $\phi$ as follows:
\[\begin{array}{rccc}
\phi:&Z(\mathbbm{k}\textbf{-}\Mod^{\mathcal{G}_x}(\mathbbm{k}[S_x],\mathbbm{k}[S_x]))&\longrightarrow &
Z(\mathbbm{k}\textbf{-}\Mod^\mathcal{G}(\mathbbm{k}[S],\mathbbm{k}[S]))\\
 &\rotatebox{90}{$\in$}& &\rotatebox{90}{$\in$}\\
 &\theta_x&\longmapsto&
(\eta_\omega)_{\omega\in\mathcal{G}_0}.
\end{array}
\]
Here, $\eta_\omega$ is the map $\eta_\omega =S(g_\omega) \circ \theta_x \circ S(g_\omega)^{-1} $.\\

\noindent
$(1)$ Well-definedness of $\phi$\\
\noindent
For each $\{\tau_\omega\}_{\omega\in\mathcal{G}_0}\in \mathbbm{k}$-$\Mod^\mathcal{G}(\mathbbm{k}[S],\mathbbm{k}[S])$ and $w\in\mathcal{G}_0$, we obtain
\[\begin{aligned}
 \eta_\omega\circ\tau_\omega 
 &=S(g_\omega)\circ\theta_x\circ S(g_\omega)^{-1}\circ S(g_\omega)\circ\tau_x\circ S(g_\omega)^{-1}\\
 &=S(g_\omega)\circ\tau_x\circ S(g_\omega)^{-1}\circ S(g_\omega)\circ\theta_x\circ S(g_\omega)^{-1}\\
 &=\tau_\omega\circ \eta_\omega.
\end{aligned}\]
\noindent
For $x,y,z\in\mathcal{G}_0$, we construct the following bijections.
\[\begin{array}{rccc}
f:&\mathcal{G}_x&\longrightarrow &
\mathcal{G}(x,y)\\
 &\rotatebox{90}{$\in$}& &\rotatebox{90}{$\in$}\\
 &u&\longmapsto&g_yu
\end{array},
\]
\[\begin{array}{rccc}
f:&\mathcal{G}(x,y)&\longrightarrow &
\mathcal{G}(y,z)\\
 &\rotatebox{90}{$\in$}& &\rotatebox{90}{$\in$}\\
 &a&\longmapsto&g_za^{-1}
\end{array}.
\]
\noindent
For each $y\in\mathcal{G}_0$ and each morphism $g:x\to y$, there exists a morphism $u:x\to x$ such that $g=g_yu$. In the following diagram, since face $A$ and $B$ are commutative, the back face is also commutative, so we have a commutative diagram 
\begin{center}
\begin{tikzpicture}[
 >=latex,
 scale=1.18,
 transform shape,
 line width=0.5pt,
 every node/.style={ inner sep=2pt},
 lab/.style={fill=white, inner sep=1.2pt, font=\small}
]
 \node (A) at (0.0, 3.2) {$S(x)$};
 \node (B) at (7.0, 3.2) {$S(y)$};
 \node (C) at (3.2, 1.9) {$S(x)$};
 \node (D) at (0.0, 0.7) {$S(x)$};
 \node (E) at (7.0, 0.7) {$S(y)$};
 \node (F) at (3.2, -0.6) {$S(x)$};

 \node at (1.5, 1.3) {$A$};
 \node at (5.3, 1.3) {$B$};
 \node at (3.45, 2.85) {$\circlearrowleft$};
 \node at (3.45, 0.25) {$\circlearrowleft$};

 \draw[->] (A) -- node[lab, above]{$S(g)$} (B);
 \draw[->] (A) -- node[lab, pos=0.55, left]{$\theta_x$}  (D);
 \draw[->] (A) -- node[lab, pos=0.6, above, sloped]{$S(u)$}  (C);
 \draw[->] (C) -- node[lab, pos=0.5, above, sloped]{$S(g_y)$}   (B);
 \draw[->] (B) -- node[lab, right, xshift=4pt]{$\eta_y$} (E);
 \draw[->] (D) -- node[lab, pos=0.6, above]{$S(g)$} (E);
 \draw[->] (C) -- node[lab, pos=0.3, left]{$\theta_x$} (F);
 \draw[->] (D) -- node[lab, pos=0.45, below left]{$S(u)$} (F);
 \draw[->] (F) -- node[lab, pos=0.45, below right]{$S(g_y)$} (E);
\end{tikzpicture}.
\end{center}

We see that, for each $y,z\in\mathcal{G}_0$ and each morphism $g:y\to z$, there exists a morphism $a:x\to y$ such that $g=g_za^{-1}$. In the following diagram, since face $A$ and $B$ are commutative, the back face is also commutative, so we have a commutative diagram 
\begin{center}
\begin{tikzpicture}[
 >=latex,
 scale=1.18,
 transform shape,
 line width=0.5pt,
 every node/.style={ inner sep=2pt},
 lab/.style={fill=white, inner sep=1.2pt, font=\small}
]
 \node (A) at (0.0, 3.2) {$S(y)$};
 \node (B) at (7.0, 3.2) {$S(z)$};
 \node (C) at (3.2, 1.9) {$S(x)$};
 \node (D) at (0.0, 0.7) {$S(y)$};
 \node (E) at (7.0, 0.7) {$S(z)$};
 \node (F) at (3.2, -0.6) {$S(x)$};

 \node at (1.5, 1.3) {$A$};
 \node at (5.3, 1.3) {$B$};
 \node at (3.45, 2.85) {$\circlearrowleft$};
 \node at (3.45, 0.25) {$\circlearrowleft$};

 \draw[->] (A) -- node[lab, above]{$S(g)$} (B);
 \draw[->] (A) -- node[lab, pos=0.55, left]{$\eta_y$}  (D);
 \draw[->] (A) -- node[lab, pos=0.6, above, sloped]{$S(a^{-1})$}  (C);
 \draw[->] (C) -- node[lab, pos=0.5, above, sloped]{$S(g_z)$}   (B);
 \draw[->] (B) -- node[lab, right, xshift=4pt]{$\eta_z$} (E);
 \draw[->] (D) -- node[lab, pos=0.6, above]{$S(g)$} (E);
 \draw[->] (C) -- node[lab, pos=0.3, left]{$\theta_x$} (F);
 \draw[->] (D) -- node[lab, pos=0.45, below left]{$S(a^{-1})$} (F);
 \draw[->] (F) -- node[lab, pos=0.45, below right]{$S(g_z)$} (E);
\end{tikzpicture}.
\end{center}
Hence, $\{\eta_\omega\}_{\omega\in\mathcal{G}_0}$ is a natural transformation from $\mathbbm{k}[S]$ to $\mathbbm{k}[S]$.\\

\noindent
$(2)$ Injectivity of $\phi$\\
\noindent
Assume that $\phi(\theta_x)=\phi(\tau_x)$. Then, for each $w\in\mathcal{G}_0$, we obtain \[\phi(\theta_x)_\omega=S(g_\omega)\circ\theta_x\circ S(g_\omega)^{-1}=S(g_\omega)\circ\tau_x\circ S(g_\omega)^{-1}=\phi(\tau_x)_\omega.\]
Consequently, $\theta_x=\tau_x$ and $\phi$ is injective.\\

\noindent
$(3)$ Surjectivity of $\phi$\\
\noindent
For each $\eta\in Z(\mathbbm{k}$-$\Mod^\mathcal{G}(\mathbbm{k}[S],\mathbbm{k}[S]))$, we have $\phi(\eta_x)=\eta$. Thus, $\phi$ is surjective.\\

\noindent
$(4)$ The ring homomorphism property of $\phi$\\
\noindent
For each $\theta_x, \tau_x\in Z(\mathbbm{k}$-$\Mod^{\mathcal{G}_x}(\mathbbm{k}[S_x],\mathbbm{k}[S_x]))$ and $w\in \mathcal{G}_0$, we obtain 
\[\begin{aligned}
 \phi(\theta_x\circ\tau_x)_\omega &=S(g_\omega)\circ\theta_x\circ\tau_x\circ S(g_\omega)^{-1}\\
 &=S(g_\omega)\circ\theta_x\circ S(g_\omega)^{-1}\circ S(g_\omega)\circ\tau_x\circ S(g_\omega)^{-1}\\
 &=\phi(\theta_x)\circ \phi (\tau_x).
\end{aligned}\]
Thus, $\phi$ is a ring homomorphism. 

From, $(1)-(4)$, we have an isomorphism of rings \[Z(\mathbbm{k}\text{-}\Mod^{\mathcal{G}}(\mathbbm{k}[S],\mathbbm{k}[S]))\cong Z(\mathbbm{k}\text{-}\Mod^{\mathcal{G}_x}(\mathbbm{k}[S\circ\incl^\mathcal{G}_{\mathcal{G}_x}],\mathbbm{k}[S\circ\incl^\mathcal{G}_{\mathcal{G}_x}])).\]
\end{proof}

\begin{theorem}\label{pd}
If $\mathcal{G}\cong\coprod^n_{i=1}\mathcal{G}_{[i]}$ is the decomposition of $\mathcal{G}$ into its connected components, then we have an isomorphism of rings
\[Z(\mathbbm{k}\text{-}\Mod^{\mathcal{G}}(\mathbbm{k}[S],\mathbbm{k}[S]))\cong \prod_{x\in\pi_0(\mathcal{G})}Z(\mathbbm{k}\text{-}\Mod^{\mathcal{G}_x}(\mathbbm{k}[S\circ\incl^\mathcal{G}_{\mathcal{G}_x}],\mathbbm{k}[S\circ\incl^\mathcal{G}_{\mathcal{G}_x}])).\]
\end{theorem}

\begin{proof}
Consider the map
\[\begin{array}{rccc}
f:&Z(\mathbbm{k}\text{-}\Mod^{\mathcal{G}}(\mathbbm{k}[S],\mathbbm{k}[S]))&\longrightarrow &
\displaystyle\prod^n_{i=1}Z(\mathbbm{k}\text{-}\Mod^{\mathcal{G}_{[i]}}(\mathbbm{k}[S\circ\incl^\mathcal{G}_{\mathcal{G}_{[i]}}],\mathbbm{k}[S\circ\incl^\mathcal{G}_{\mathcal{G}_{[i]}}]))\\
 &\rotatebox{90}{$\in$}& &\rotatebox{90}{$\in$}\\
 &\theta&\longmapsto&
(\theta\circ \incl^\mathcal{G}_{\mathcal{G}_{[i]}})^n_{i=1}.
\end{array}
\]
It follows from the construction that $f$ is a ring isomorphism. 
By Theorem \ref{center}, we have an isomorphism of rings
\[Z(\mathbbm{k}\text{-}\Mod^{\mathcal{G}}(\mathbbm{k}[S],\mathbbm{k}[S]))\cong \prod_{x\in\pi_0(\mathcal{G})}Z(\mathbbm{k}\text{-}\Mod^{\mathcal{G}_x}(\mathbbm{k}[S\circ\incl^\mathcal{G}_{\mathcal{G}_x}],\mathbbm{k}[S\circ\incl^\mathcal{G}_{\mathcal{G}_x}])).\]
\end{proof}

\section{Yoshida algebra for groupoids}
In this section, we introduce Yoshida algebras for finite groupoids and describe their relationship between the crossed Burnside ring of a groupoid and the groupoid algebra.
\begin{definition}
Let $\mathcal{G}\cong \coprod^n_{i=1}\mathcal{G}[i]$ be the decomposition of $\mathcal{G}$ into its connected components. For each $i\in \{1,\ldots,n\}$, fix a base object $x_i\in\mathcal{G}[i]_0$, identify $\mathcal{G}[i]$ with $\mathcal{G}_{x_i}\times \Pair\;{(\mathcal{G}[i]_0)}$ via Proposition \ref{gpdst} and write $\lambda_{x,y}$ for an element of $\mathcal{G}[i]_1$ corresponding to $(x,y)\in\Pair\;({\mathcal{G}[i]_0})_1$. In particular, $\lambda_{x,x}=\id_x$ and the composition satisfies $\lambda_{y,z}\lambda_{x,y}=\lambda_{x,z}$. Let $\Omega:\mathcal{G}\to \set$ be a functor defined by the following conditions: 
\begin{enumerate}
\item For an object $x\in \mathcal{G}_0$, define $\Omega (x):=\coprod_{H\le \mathcal{G}_x}\mathcal{G}_x/H,$
\item For a morphism $g:x\to y$, define $\Omega(g):\Omega(x)\to \Omega(y)$ by $\langle H,kH\rangle \mapsto \langle \lambda_{x,y}H\lambda_{x,y}^{-1}, gkH\lambda_{x,y}^{-1}\rangle$.
\end{enumerate}
\end{definition}

\begin{definition}
The ring $\mathbbm{k}$-$\Mod^{\mathcal{G}}(\mathbbm{k}[\Omega^2],\mathbbm{k}[\Omega^2])$ is called \textbf{Yoshida algebra}, and we denote it by $Y_\mathbbm{k}(\mathcal{G})$.
\end{definition}

\begin{remark}
The algebra $Y_\mathbbm{k}(G)$ of a finite group $G$ over $\mathbbm{k}$ was introduced by Yoshida in \cite{Yos83}, and was first referred to as a Yoshida algebra by Bouc in \cite{Bou97}.
\end{remark}

By combining the results on the cohomological Mackey algebra $\cou_\mathbbm{k}(G))$ in \cite{Rog15}, we obtain the following theorem.

\begin{theorem}\label{Yosiso}
If $G$ is a finite group, then $Z(\mathbbm{k}[G])\cong Z(Y_\mathbbm{k}(G))$.
\end{theorem}

\begin{proof}
Combining the isomorphism $Z(\mathbbm{k}[G])\cong Z(\cou_\mathbbm{k}(G))$ from \cite[Proposition 3.4]{Rog15} and from $\cou_\mathbbm{k}(G)\cong \End_\mathbbm{k}(\Omega^2)$\cite[Theorem 2.12]{Rog15}, we conclude that $Z(\mathbbm{k}[G])\cong Z(Y_\mathbbm{k}(G))$.
\end{proof}

\begin{theorem}\label{yd1}
If $\mathcal{G}$ is connected and $x\in \mathcal{G}_0$, then two rings $Z(Y_\mathbbm{k}(\mathcal{G}))$ and $Z(Y_\mathbbm{k}(\mathcal{G}_x))$ are isomorphic.
\end{theorem}

\begin{proof}
This follows from Theorem \ref{center}.
\end{proof}

\begin{theorem}\label{yd2}
Two rings $Z(Y_\mathbbm{k}(\mathcal{G}))$ and $\prod_{x\in\pi_0(\mathcal{G})}Z(Y_\mathbbm{k}(\mathcal{G}_x))$ are isomorphic.
\end{theorem}

\begin{proof}
This follows from Theorem \ref{pd}.
\end{proof}

\begin{theorem}\label{ZY}
If $\mathcal{G}$ is connected, then two rings $Z(\mathbbm{k}[\mathcal{G}])$ and $Z(Y_\mathbbm{k}(\mathcal{G}))$ are isomorphic.
\end{theorem}

\begin{proof}
If $x\in\mathcal{G}_0$, then $Z(\mathbbm{k}[\mathcal{G}])\cong Z(\mathbbm{k}[\mathcal{G}_x])$ by Corollary \ref{morita} and $Z(Y_\mathbbm{k}(\mathcal{G}))\cong Z(Y_\mathbbm{k}(\mathcal{G}_x))$ by Theorem \ref{yd1}. Since $Z(\mathbbm{k}[\mathcal{G}_x])\cong Z(Y_\mathbbm{k}(\mathcal{G}_x))$ by Theorem \ref{Yosiso}, we obtain $Z(\mathbbm{k}[\mathcal{G}])\cong Z(Y_\mathbbm{k}(\mathcal{G}))$.
\end{proof}

The same result holds even when the groupoid is not connected.
\begin{theorem} \label{isocenter}
Two rings $Z(\mathbbm{k}[\mathcal{G}])$ and $Z(Y_\mathbbm{k}(\mathcal{G}))$ are isomorphic.
\end{theorem}

\begin{proof}
If $\mathcal{G}\cong\coprod^n_{i=1}\mathcal{G}_{[i]}$ is the decomposition of $\mathcal{G}$ into its connected components, then by Theorem \ref{gd} and Theorem \ref{pd}, we obtain the isomorphism $Z(Y_\mathbbm{k}(\mathcal{G}))\cong 
\prod^n_{i=1}Z(Y_\mathbbm{k}({\mathcal{G}}_{[i]}))$ and 
$Z(\mathbbm{k}[\mathcal{G}])\cong \prod^n_{i=1}Z(\mathbbm{k}[\mathcal{G}_{[i]}])$ respectively. 
By Theorem \ref{ZY}, there exists a ring isomomorphism
$Z(Y_\mathbbm{k}(\mathcal{G}_{[i]}))\to Z(\mathbbm{k}[\mathcal{G}_{[i]}])$, which we denote by $\phi_i$.
Consider the map
\[\begin{array}{rccc}
f:&\prod^n_{i=1}Z(Y_\mathbbm{k}({\mathcal{G}}_{[i]}))&\longrightarrow &
\prod^n_{i=1}Z(\mathbbm{k}[\mathcal{G}_{[i]}])\\
 &\rotatebox{90}{$\in$}& &\rotatebox{90}{$\in$}\\
 &(x_i)^n_{i=1}&\longmapsto&
(\phi_i(x_i))^n_{i=1}.
\end{array}
\]
It follows from the construction that $f$ is a ring isomorphism. Therefore, two rings $Z(\mathbbm{k}[\mathcal{G}])$ and $Z(Y_\mathbbm{k}(\mathcal{G}))$ are isomorphic.
\end{proof}

\begin{theorem}\label{isocenter2}
There exists a surjective ring homomorphism from $B_\mathbbm{k}^c(\mathcal{G})$ to $Z(\mathbbm{k}[\mathcal{G}])$.
\end{theorem}

\begin{proof}
If $\mathcal{G}\cong\coprod^n_{i=1}\mathcal{G}_{[i]}$ is the decomposition of $\mathcal{G}$ into its connected components, then by Theorem \ref{gd} and Theorem \ref{cd}, we obtain the isomorphism  $B^c_\mathbbm{k}(\mathcal{G})\cong 
\prod^n_{i=1}B^c_\mathbbm{k}({\mathcal{G}}_{[i]})$ and 
$Z(\mathbbm{k}[\mathcal{G}])\cong \prod^n_{i=1}Z(\mathbbm{k}[\mathcal{G}_{[i]}])$ respectively. 
By Lemma \ref{stt}, there exists a surjective ring homomorphism
$B^c_\mathbbm{k}(\mathcal{G}_{[i]})\to Z(\mathbbm{k}[\mathcal{G}_{[i]}])$, which we denote by $\phi_i$.
Consider the map
\[\begin{array}{rccc}
f:&\prod^n_{i=1}B^c_\mathbbm{k}({\mathcal{G}}_{[i]})&\longrightarrow &
\prod^n_{i=1}Z(\mathbbm{k}[\mathcal{G}_{[i]}])\\
 &\rotatebox{90}{$\in$}& &\rotatebox{90}{$\in$}\\
 &(x_i)^n_{i=1}&\longmapsto&
(\phi_i(x_i))^n_{i=1}.
\end{array}
\]
It follows from the construction that $f$ is a surjective ring homomorphism. 
Therefore, there exists a surjective ring homomorphism from $B_\mathbbm{k}^c(\mathcal{G})$ to $Z(\mathbbm{k}[\mathcal{G}])$.
\end{proof}

\begin{corollary}
There exists a surjective ring homomorphism from $B_\mathbbm{k}^c(\mathcal{G})$ to $Z(Y_\mathbbm{k}(\mathcal{G}))$.
\end{corollary}

\begin{proof}
This follows from Theorem \ref{isocenter} and Theorem \ref{isocenter2}.
\end{proof}

\section*{Acknowledgement}
The author sincerely thanks his supervisor, Professor Fumihito Oda, for his support and valuable comments.

\end{document}